\documentclass[%
 aip,
 jmp,
 amsmath,amssymb,
 preprint,
]{revtex4-1}

\usepackage{graphicx}
\usepackage{dcolumn}
\usepackage{bm}

\usepackage[utf8]{inputenc}
\usepackage[T1]{fontenc}
\usepackage{mathptmx}
\usepackage{etoolbox}

\usepackage{amsfonts}
\usepackage{amsthm}
\newtheorem{theorem}{Theorem}

\newtheorem{corollary}[theorem]{Corollary}

\newtheorem{definition}[theorem]{Definition}

\newtheorem{lemma}[theorem]{Lemma}

\newtheorem{proposition}[theorem]{Proposition}
\newtheorem{remark}[theorem]{Remark}

\makeatletter
\def\@email#1#2{%
 \endgroup
 \patchcmd{\titleblock@produce}
  {\frontmatter@RRAPformat}
  {\frontmatter@RRAPformat{\produce@RRAP{*#1\href{mailto:#2}{#2}}}\frontmatter@RRAPformat}
  {}{}
}%
\makeatother

\begin{document}


\title{A Constructive Approach to $q$-Gaussian Distributions: $\alpha$-Divergence as Rate Function and Generalized de Moivre-Laplace Theorem\thanks{The material in this paper was presented in part at the 2014 International Symposium on Information Theory and Its Applications (ISITA2014) \cite{HS14}.}}

\author{Hiroki Suyari}
\email[H. Suyari, Corresponding Author: ]{suyari@faculty.chiba-u.jp, suyarilab@gmail.com}
\affiliation{Graduate School of Informatics, Chiba University, 1-33, Yayoi-cho, Inage-ku, Chiba 263-8522, Japan}

\author{Antonio M. Scarfone}
\email[A. M. Scarfone: ]{antonio.scarfone@polito.it}
\affiliation{Istituto dei Sistemi Complessi (ISC-CNR) c/o, Politecnico di Torino, Corso Duca degli Abruzzi 24, Torino I-10129, Italy}

\date{\today}

\begin{abstract}
The Large Deviation Principle (LDP) and the Central Limit Theorem (CLT) are central pillars of probability theory.
While their formulations are established under the i.i.d. assumption, the probabilistic foundation for power-law distributions has primarily evolved through descriptive models or variational principles, rather than a constructive derivation comparable to the classical binomial process.
This paper establishes a constructive probabilistic framework for power-law distributions, proceeding from the nonlinear differential equation $dy/dx = y^q$ without assuming a specific distribution a priori.
We build the algebraic and combinatorial foundations, which lead to a generalized binomial distribution based on finite counting.
We prove the LDP for this generalized binomial distribution in the regime $0 < q < 1$, demonstrating that the $\alpha$-divergence is identified as the rate function, and clarify the breakdown of this macroscopic scaling for heavier tails ($q > 1$).
This result connects our constructive framework to the structures of information geometry.
Furthermore, we prove a generalized de Moivre-Laplace theorem, showing that the generalized binomial distribution converges to a heavy-tailed limit distribution (the $q$-Gaussian distribution).
We derive that the scaling law follows the order of $n^{q/2}$ as a consequence of the underlying nonlinearity.
These analytical results are numerically verified for distinct values of $q \in (0, 2)$.
This framework provides a constructive basis that unifies the shift-invariant exponential family and the rescaling-invariant power-law family.
\end{abstract}


\maketitle


\section{Introduction}

The Large Deviation Principle (LDP) provides insights into the asymptotic behavior of rare events in stochastic phenomena, ranging from probability theory and statistical mechanics to information theory \cite{Varadhan66, Ellis85, CoverThomas06, Touchette09}.
Classic LDP results rely on the assumption of independent and identically distributed (i.i.d.) random variables.
While this assumption simplifies the analysis, yielding exponential decay of probabilities, it is restrictive for systems exhibiting power-law behaviors.
To extend the LDP framework beyond the i.i.d. setting, we reconsider the algebraic structure governing the probability laws.
The exponential family, which characterizes i.i.d. statistics, follows the \textit{exponential law}: $\exp(x)\exp(a) = \exp(x+a)$.
This implies that multiplication corresponds to a \textit{shift} operation in the argument:
\begin{equation}
x \mapsto x+a.
\label{shift}
\end{equation}
Under this shift property, the scale unit of measurement is invariant.
In contrast, power-law distributions do not obey this shift property.
Instead, they follow a \textit{rescaling} law:
\begin{equation}
x \mapsto ax \quad (a>0).
\label{rescaling}
\end{equation}
Here, the scale unit itself transforms.
This paper unifies these two mathematical structures---shift and rescaling---within a single framework.
Motivated by this scale-theoretic perspective, we adopt the nonlinear differential equation as our starting point:
\begin{equation}
\frac{dy}{dx} = y^q.
\label{nlde}
\end{equation}
This equation interpolates between the shift operation ($q=1$, recovering the exponential function) and the rescaling operation ($q \neq 1$, generating power functions).

\subsection{Contributions: A Constructive Approach via $\alpha$-Divergence}
While descriptive models of generalized statistics are known, their probabilistic foundation requires further formalization.
Existing derivations typically rely on entropy maximization and variational principles \cite{Ts88, TLSM95, TMP98, PT99} or the generalized law of errors \cite{Suyari04-LawofError}.
While these methods identify the $q$-Gaussian distribution as a macroscopic state, they do not detail the microscopic mechanism of how such a distribution arises from elementary random processes.
This paper presents a constructive approach, deriving the $q$-Gaussian distribution and the generalized LDP from the equation (\ref{nlde}), without relying on variational principles or specific functional assumptions.
Our derivation proceeds from the algebraic structure of finite counting, to a generalized binomial distribution, and finally proving its convergence to the $q$-Gaussian distribution by establishing a generalized de Moivre-Laplace theorem.
This paper provides a derivation of the \textit{$\alpha$-divergence} as the rate function for the generalized binomial process.
The $\alpha$-divergence is a concept in information geometry, known to be the unique divergence belonging to both the $f$-divergence and Bregman divergence classes \cite{Amari1985, Csiszar91, Amari2009, Amari2016}.
We demonstrate that the $\alpha$-divergence is derived as a consequence of the combinatorics rooted in the nonlinear differential equation \eqref{nlde}.
This result distinguishes our generalized binomial distribution from alternative generalizations, establishing it as a probabilistic model compatible with information geometry.

Our logical derivation proceeds as follows:
\begin{enumerate}
    \item[(i)] Starting from the equation $dy/dx=y^q$, we define the $q$-logarithm and the algebraic structure based on the $q$-product.
    \item[(ii)] We derive the $q$-Stirling's formula and define the $q$-binomial distribution $b_q(k;n,r)$ via combinatorial counting.
    \item[(iii)] We prove the LDP for this distribution for $0 < q < 1$, identifying the rate function as the $\alpha$-divergence (where $q=\frac{1-\alpha}{2}$), and demonstrate the failure of this standard LDP scaling for $q > 1$.
    This connects our constructive approach to the structures of information geometry.
    \item[(iv)] Finally, we prove the Generalized de Moivre-Laplace Theorem, which serves as a fundamental realization of the $q$-Central Limit Theorem ($q$-CLT).
    We demonstrate that the $q$-binomial distribution converges to the $q$-Gaussian distribution with a fluctuation scaling of order $n^{q/2}$, as a consequence of the nonlinearity $q$.
    \item[(v)] We numerically verify these analytical results for distinct values of $q \in (0, 2)$, confirming the asymptotic convergence and the derived scaling law.
\end{enumerate}

\subsection{Related Works and Organization}
Several approaches have been proposed to extend the Central Limit Theorem to complex or non-additive systems.
Notably, Umarov et al. \cite{Umarov2008} proposed a $q$-generalized CLT utilizing the $q$-Fourier transform.
However, their transform-based method involves analytical complexities regarding the convergence of functional sequences and does not provide a microscopic generative mechanism for the stochastic process.
In contrast, the present study establishes a constructive derivation rooted in elementary probability theory, specifically by extending the de Moivre-Laplace theorem from finite combinatorial counting.
Furthermore, our formulation not only derives the $q$-Gaussian distribution but also identifies the \textit{$\alpha$-divergence} as the exact rate function governing the large deviation principle---a structural correspondence that is fundamentally inaccessible within the $q$-Fourier framework.
Our work also differs from numerical or phenomenological studies such as \cite{RuizTsallis2012}.
While \cite{RuizTsallis2012} numerically investigates the LDP for correlated systems assuming $q$-Gaussians a priori, our work is analytical and derives the distribution and the rate function from first principles.
Similarly, unlike \cite{NaudtsSuyari2015} which applies $q$-exponentials to \textit{i.i.d.} assumptions, we construct a consistent \textit{non-i.i.d.} model rooted in the fundamental algebra.
In recent years, the formalism of large deviation theory for complex systems and strongly correlated processes has been extensively investigated from various perspectives \cite{Tirnakli2021, Tirnakli2022}. 
While these studies provide significant insights into the entropic extensivity and numerical behavior of such systems, the present work focuses on establishing a rigorous mathematical foundation through a constructive approach to the $q$-Gaussian distribution.

A preliminary version of this work, which establishes the results up to the large deviation principle (Sections \ref{sec:preliminaries}--\ref{sec:alpha_divergence_LDP}), was presented at the 2014 International Symposium on Information Theory and Its Applications (ISITA 2014) \cite{HS14}.
The present paper extends this foundational framework to establish the proof of the $q$-generalized Central Limit Theorem and the derivation of the scaling law (Section \ref{sec:q-CLT}).
The remainder of this paper is organized as follows.
Section \ref{sec:preliminaries} reviews the mathematical preliminaries and algebraic foundations.
Section \ref{sec:generalized_binomial} introduces the generalized binomial distribution derived from the refined $q$-Stirling's formula.
Section \ref{sec:alpha_divergence_LDP} establishes the connection to the $\alpha$-divergence and presents the large deviation estimate.
Section \ref{sec:q-CLT} provides the proof of the generalized Central Limit Theorem.
Section \ref{sec:numerical_verification} presents the numerical verification of the analytical results.
Section \ref{sec:conclusion} concludes the paper.


\section{Mathematical Preliminaries}
\label{sec:preliminaries}

This section reviews the algebraic structures and combinatorial formulas derived from the fundamental nonlinear differential equation, which serve as the basis for the subsequent analysis.

\subsection{$q$-Logarithm Representation and Rescaling Invariance}

The nonlinear differential equation (\ref{nlde}) introduces a generalized algebraic structure \cite{Tsallis04-PhysicaD}.
We define the fundamental functions that linearize this dynamics.

\begin{definition}[$q$-logarithm, $q$-exponential]
For $q \in \mathbb{R}$, the $q$-logarithm $\ln _{q}x:\mathbb{R}^{+}\rightarrow \mathbb{R}$ and the $q$-exponential $\exp _{q}\left( x\right) :\mathbb{R}\rightarrow \mathbb{R}$ (defined for $1+\left( 1-q\right) x>0$) are defined as:
\begin{align}
\ln _{q}x &:= \frac{x^{1-q}-1}{1-q}, \label{lnq} \\
\exp _{q}\left( x\right) &:= \left[ 1+\left( 1-q\right) x\right] ^{\frac{1}{1-q}}.
\label{expq}
\end{align}
\end{definition}

Using (\ref{lnq}), Eq.~(\ref{nlde}) is transformed into the linear form $d(\ln_{q}y)/dx=1$, yielding the solution $\ln_{q}y = x + \ln_{q}C_{0}$, where $C_{0}>0$ is determined by the initial condition.
This solution is rewritten as $y/C_{0} = \exp_{q}(x/C_{0}^{1-q})$. The dynamics exhibit the following scaling properties.

\begin{proposition}[Rescaling Invariance]
The nonlinear differential equation (\ref{nlde}) is invariant under the rescaling:
\begin{equation}
\tilde{y} := \frac{y}{C_{0}},\quad \tilde{x} := \frac{x}{{C_{0}^{1-q}}}.
\label{rescaling2}
\end{equation}
\end{proposition}

\begin{proposition}[Equivalence of Shift and Rescaling]
A shift $x\mapsto x+c$ in $y=\exp_{q}\left( x\right)$ is equivalent to a simultaneous rescaling of the axes:
\begin{equation}
y^{\prime } := \frac{y}{\exp_{q}\left( c\right) },\quad x^{\prime } := \frac{x}{\left( \exp_{q}\left( c\right) \right)^{1-q}},
\label{rescaling3}
\end{equation}
recovering the form $y^{\prime } = \exp_{q}\left( x^{\prime }\right)$.
\end{proposition}

The shift and rescaling invariance of the $q$-exponential representation induces scale variations in sequential observations when $q \neq 1$, leading to non-uniqueness in formulating probability distributions.
To avoid this ambiguity and maintain a fixed scale unit, we exclusively employ the $q$-logarithm representation throughout this paper \cite{MSW2019, SMS20}.

\subsection{Tsallis Entropy and $q$-Stirling's Formula}

We outline the combinatorial foundation. Building upon the $q$-algebra \cite{Nivanen2003, Borges2004}, we define the $q$-factorial and related combinatorial structures \cite{Suyari06}.

\begin{definition}[$q$-product \cite{Nivanen2003, Borges2004}]
For $x,y\in \mathbb{R}^{+}$ satisfying $x^{1-q}+y^{1-q}-1>0$, the $q$-product is defined as:
\begin{equation}
x\otimes _{q}y := \left[ x^{1-q}+y^{1-q}-1\right] ^{\frac{1}{1-q}}.
\end{equation}
\end{definition}

\begin{definition}[$q$-factorial \cite{Suyari06}]
Based on the $q$-product, the $q$-factorial for $n \in \mathbb{N}$ is defined as:
\begin{equation}
n!_{q} := 1\otimes _{q}2\otimes _{q}\cdots \otimes _{q}n.
\end{equation}
\end{definition}

In the $q$-logarithm representation, the $q$-factorial transforms into the exact sum $\ln_q n!_q = \sum_{k=1}^n \ln_q k$.
The asymptotic approximations of this sum, the formulation of the $q$-multinomial coefficient, and its exact relation to Tsallis entropy were established in \cite{Suyari06}.

\begin{proposition}[$q$-Stirling's Formula]
For $0 < q < 2$, the leading-order and refined approximations of $\ln_{q}n!_{q}$ are respectively given by:
\begin{align}
\ln _{q}n!_{q} &= \frac{n}{2-q}\ln _{q}n - \frac{n}{2-q} + O\left( \ln _{q}n\right), \label{prop:leading_q_stirling} \\
\ln _{q}n!_{q} &= \left( \frac{n}{2-q}+\frac{1}{2}\right) \ln _{q}n - \frac{n}{2-q} + c_{q} + O\left( n^{-q}\right), \label{prop:refined_q_stirling}
\end{align}
where $c_{q}$ is a constant dependent only on $q$.
\end{proposition}

The refined form (\ref{prop:refined_q_stirling}) is essential for deriving the exact rate function in the subsequent sections.

\begin{definition}[$q$-multinomial coefficient]
The $q$-logarithm of the $q$-multinomial coefficient is defined as:
\begin{equation}
\ln _{q}\left( \begin{array}{ccc} & n & \\ n_{1} & \cdots & n_{k} \end{array} \right)_{q} := \ln _{q}n!_{q} - \sum_{i=1}^{k} \ln _{q}n_{i}!_{q},
\label{def:q-multinomial_log}
\end{equation}
where $n=\sum_{i=1}^{k}n_{i}$.
\end{definition}

\begin{proposition}[Combinatorial Determination of Tsallis Entropy]
For $0 < q < 2$, the asymptotic relation to Tsallis entropy $S_{q}^{\text{Tsallis}}(P) := \frac{1-\sum p_i^q}{q-1}$ is:
\begin{equation}
\ln _{q}\left( \begin{array}{ccc} & n & \\ n_{1} & \cdots & n_{k} \end{array} \right)_{q} = \frac{n^{2-q}}{2-q} S_{2-q}^{\text{Tsallis}}\left( \frac{n_{1}}{n},\cdots ,\frac{n_{k}}{n}\right) + O(\ln _{q}n).
\label{eq:one-to-one_Tsallis}
\end{equation}
\end{proposition}


\section{The Generalized Binomial Distribution Derived from the Refined $q$-Stirling's Formula}
\label{sec:generalized_binomial}

This section introduces the generalized binomial distribution by applying the refined $q$-Stirling's formula to establish the exact asymptotic relationship between the $q$-binomial coefficient and Tsallis entropy.

\subsection{Relation between Refined $q$-Stirling's Formula and Tsallis Entropy}

By utilizing the refined $q$-Stirling's formula (\ref{prop:refined_q_stirling}), we establish an asymptotic relationship between the $q$-binomial coefficient and Tsallis entropy.
This relationship leads to the formal definition of the generalized binomial distribution.

\begin{proposition}
\label{prop:q-binomial_entropy}
For $0 < q < 2$, the $q$-logarithm of the $q$-binomial coefficient satisfies the following asymptotic expansion:
\begin{equation}
\ln _{q}\left( 
\begin{array}{l}
n \\ 
k%
\end{array}%
\right) _{q} = -c_{q} + \frac{1}{2}\left( \ln _{q}n - \ln _{q}k - \ln _{q}\left( n-k\right) \right) + \frac{n^{2-q}}{2-q}S_{2-q}^{\text{Tsallis}}\left( \frac{k}{n}, 1-\frac{k}{n}\right) + O(n^{-q}).
\label{q-log_q-bino-coef}
\end{equation}
\end{proposition}

The proof follows from a direct computation using (\ref{prop:refined_q_stirling}).
Evaluating the standard case ($q=1$) in (\ref{q-log_q-bino-coef}) yields:
\begin{align}
\ln \left( 
\begin{array}{l}
n \\ 
k%
\end{array}%
\right) & = -\ln \sqrt{2\pi } + \frac{1}{2}\ln \frac{n}{k\left( n-k\right) } + nS_{1}\left( \frac{k}{n}, 1-\frac{k}{n}\right) + O(n^{-1})
\label{stand-log_stand-bino-coef-1} \\
& = \ln \left( \frac{1}{\sqrt{2\pi }}\sqrt{\frac{n}{k\left( n-k\right) }} \right) + \left( -k\ln \frac{k}{n} - \left( n-k\right) \ln \left( 1-\frac{k}{n}\right) \right) + O(n^{-1}).
\label{stand-log_stand-bino-coef}
\end{align}%
Exponentiating the leading terms provides the rigorous asymptotic expansion:
\begin{equation}
\left( 
\begin{array}{l}
n \\ 
k%
\end{array}%
\right) \left( \frac{k}{n}\right) ^{k}\left( 1-\frac{k}{n}\right) ^{n-k} = \frac{1}{\sqrt{2\pi }}\sqrt{\frac{n}{k\left( n-k\right) }} \left( 1 + O\left(\frac{1}{n}\right) \right).
\label{special form of the standard}
\end{equation}%
The left-hand side corresponds to the standard binomial distribution with the substitution $r=\frac{k}{n}$.
Here, the entropy term
\begin{equation}
nS_{1}\left( \frac{k}{n}, 1-\frac{k}{n}\right) = -k\ln \frac{k}{n} - \left( n-k\right) \ln \left( 1-\frac{k}{n}\right)
\label{last term q=1}
\end{equation}%
in (\ref{stand-log_stand-bino-coef-1}) exactly corresponds to the logarithmic probability term via the replacement $r=\frac{k}{n}$:
\begin{equation}
-\ln \left( r^{k}(1-r)^{n-k} \right) = -k\ln r - \left( n-k\right) \ln \left( 1-r\right).
\label{replacement r=k/n}
\end{equation}%

Applying this structural correspondence to the generalized case (\ref{q-log_q-bino-coef}), the last term is explicitly evaluated as:
\begin{equation}
\frac{n^{2-q}}{2-q}S_{2-q}^{\text{Tsallis}}\left( \frac{k}{n}, 1-\frac{k}{n}\right) = \frac{1}{2-q}\left( -k^{2-q}\ln _{2-q}\frac{k}{n} - \left( n-k\right) ^{2-q}\ln _{2-q}\left( 1-\frac{k}{n}\right) \right).
\end{equation}%
Substituting $r=\frac{k}{n}$ into this expression yields the exact formulation of the generalized binomial distribution.

\subsection{Definition of the Generalized Binomial Distribution}

\begin{definition}[$q$-binomial distribution]
\label{def:q-binomial distribution}
For given $n,k \in \mathbb{N}$ (with $k \leq n$) and $r\in \left( 0,1\right)$, the $q$-logarithm of the $q$-binomial distribution $b_{q}\left( k;n,r\right)$ is defined by:
\begin{equation}
\ln _{q}b_{q}\left( k;n,r\right) := \ln _{q}\left( 
\begin{array}{l}
n \\ 
k%
\end{array}%
\right) _{q} + \frac{1}{2-q}\left( k^{2-q}\ln _{2-q}r + \left( n-k\right)^{2-q}\ln _{2-q}\left( 1-r\right) \right) + \ln _{q}C_{q},
\label{def_lnq_q-binomial distribution}
\end{equation}%
where $C_q$ is a normalization constant ensuring $\sum_{k=0}^{n}b_{q}\left( k;n,r\right) = 1$, with $C_{q}>0$ and $C_{1}=1$.
\end{definition}

The $q$-logarithm representation in (\ref{def_lnq_q-binomial distribution}) ensures analytical tractability.
This formulation incorporates the scaling effect via $C_{q}^{1-q}$, which vanishes in the limit $q\to 1$.
Note that this formulation can be extended to define a $q$-multinomial distribution for multiple variables.
The generalized binomial distribution yields two fundamental results:
(i) the emergence of the $\alpha$-divergence as the rate function in the associated Large Deviation Principle, and
(ii) the generalization of the de Moivre-Laplace theorem to power-law distributions.
We demonstrate these results in the subsequent sections.

While numerous generalizations of the binomial distribution have been proposed in the literature, the formulation in Definition \ref{def:q-binomial distribution} is uniquely constrained by the algebraic structure of the $q$-logarithm and the refined $q$-Stirling's formula.
This construction is analytically justified in the subsequent sections: it yields the $\alpha$-divergence as the rate function in the Large Deviation Principle (Section \ref{sec:alpha_divergence_LDP}), and it converges to the $q$-Gaussian distribution in the generalized Central Limit Theorem (Section \ref{sec:q-CLT}).
Therefore, this definition serves as the mathematically consistent formulation required to unify the scale-invariant probability laws with the structures of information geometry.


\section{$\alpha$-Divergence and Large Deviation Estimate Derived from the Generalized Binomial Distribution}
\label{sec:alpha_divergence_LDP}

\subsection{$\alpha$-Divergence Derived from the Generalized Binomial Distribution}

The $q$-divergence is derived from the definition (\ref{def_lnq_q-binomial distribution}) of the $q$-binomial distribution.

\begin{theorem}[$q$-divergence derived from the $q$-binomial distribution]
\label{thm:q-divergence_derived}
For the $q$-binomial distribution $b_{q}\left( k;n,r\right)$ defined by (\ref{def_lnq_q-binomial distribution}), we obtain the following asymptotic expansion for large $n$:
\begin{equation}
\ln _{q}b_{q}\left( k;n,r\right) = -\frac{n^{2-q}}{2-q}D_{2-q}\left( \mathbf{p} \| \mathbf{r} \right) + \ln _{q}C_{q} + O(\ln_q n),
\label{qGBin-qdiv}
\end{equation}%
where $D_{q}\left( \mathbf{p} \| \mathbf{r} \right)$ is the $q$-divergence defined by
\begin{equation}
D_{q}\left( \mathbf{p} \| \mathbf{r} \right) :=\sum_{i=0}^{1} p_{i}\ln _{2-q}\frac{p_{i}}{r_{i}}=\frac{1-\sum\limits_{i=0}^{1}p_{i}^{q}r_{i}^{1-q}}{1-q},
\end{equation}%
with the probability distributions $\mathbf{p} :=\left( \frac{k}{n},1-\frac{k}{n}\right)$ and $\mathbf{r} :=\left( r,1-r\right)$.
\end{theorem}

\begin{proof}
The proof proceeds by direct algebraic evaluation. 
Taking the $q$-logarithm of the $q$-binomial distribution (\ref{def_lnq_q-binomial distribution}), we expand the combinatorial factor asymptotically for large $n$. 
By substituting the empirical distribution $\mathbf{p} = (k/n, 1-k/n)$ and rearranging the terms via the identity (\ref{eq:one-to-one_Tsallis}), the leading terms factorize into the scaled divergence $-\frac{n^{2-q}}{2-q}D_{2-q}\left( \mathbf{p} \| \mathbf{r} \right)$ and the normalization component $\ln_{q}C_{q}$. 
The remaining lower-order fluctuations are absorbed into the residual $O(\ln_q n)$.
\end{proof}

\begin{remark}
The derivation directly extends to the multinomial case. By applying the same procedure to multiple variables, the generalized $q$-divergence $D_{2-q}(p\|r)$ emerges as the corresponding rate function, which recovers the standard Kullback-Leibler divergence in the limit $q\rightarrow 1$.
This generalized divergence is algebraically related to the $\alpha$-divergence \cite{Oha07}.
\end{remark}

\begin{proposition}
The $\alpha$-divergence $D^{\left( \alpha\right) }\left( \mathbf{p} \| \mathbf{r} \right)$ defined by
\begin{equation}
D^{\left( \alpha\right) }\left( \mathbf{p} \| \mathbf{r} \right) :=\left\{ 
\begin{array}{ll}
\frac{4}{1-\alpha^{2}}\left( 1-\sum\limits_{i}p_{i}^{\frac{1-\alpha}{2}}r_{i}^{\frac{1+\alpha}{2}}\right) & \quad\left( \alpha\neq\pm1\right) \\ 
\sum\limits_{i}r_{i}\ln\frac{r_{i}}{p_{i}} & \quad\left( \alpha=1\right) \\ 
\sum\limits_{i}p_{i}\ln\frac{p_{i}}{r_{i}} & \quad\left( \alpha=-1\right)%
\end{array}
\right.
\label{alpha-divergence}
\end{equation}
relates to the $q$-divergence via:
\begin{equation}
D^{\left( \alpha\right) }\left( \mathbf{p} \| \mathbf{r} \right) =\frac{1}{q}D_{q}\left( \mathbf{p} \| \mathbf{r} \right) \quad\left( q\neq0,1\right),
\label{alpha-q-divergence-relation}
\end{equation}
where $q=\frac{1-\alpha}{2}$ ($\alpha\neq\pm1$).
\end{proposition}

This equivalence implies that the $\alpha$-divergence is obtained from the asymptotic behavior of the $q$-binomial distribution.

\begin{corollary}
By applying the parameter transformation $q = \frac{1-\alpha}{2}$ to the $q$-binomial distribution $b_{q}\left( k;n,r\right)$, we establish the connection to the $\alpha$-geometry as follows:
\begin{equation}
\ln_{\frac{1-\alpha}{2}} b_{\frac{1-\alpha}{2}}\left( k;n,r\right) = -n^{\frac{3+\alpha}{2}} D^{(-2-\alpha)}\left( \mathbf{p} \| \mathbf{r} \right) + \ln_{\frac{1-\alpha}{2}} C_{\frac{1-\alpha}{2}} + O\left(\ln_{\frac{1-\alpha}{2}} n\right),
\end{equation}%
where $D^{(-2-\alpha)}\left( \mathbf{p} \| \mathbf{r} \right)$ represents the corresponding $\alpha$-divergence of index $-2-\alpha$, and $\mathbf{p} = (k/n, 1-k/n)$, $\mathbf{r} = (r, 1-r)$.
\end{corollary}

Historically, the $\alpha$-divergence \cite{Cher52} and $q$-divergence \cite{Ts98} were introduced independently.
The result above connects these two concepts through the constructive definition of the generalized binomial distribution.

\subsection{Large Deviation Estimate in Power-Law Distributions}

Using (\ref{qGBin-qdiv}), we present the large deviation estimate in terms of the $q$-divergence.
First, we establish the monotonicity of the distribution in the tail region.

\begin{lemma}
\label{lemma:monotonicity}
For $1\leq k\leq nr$, and for sufficiently large $n$, the distribution satisfies the monotonicity condition:
\begin{equation}
\ln _{q}b_{q}\left( k;n,r\right) \geq \ln _{q}b_{q}\left( k-1;n,r\right).
\end{equation}
\end{lemma}

\begin{proof}
We analyze the continuous interpolation of the generalized binomial distribution.
Let $x$ be a continuous variable in $[0, n]$.
We define the function $\mathcal{L}_q(x)$ based on the leading terms of (\ref{def_lnq_q-binomial distribution}):
\begin{align}
\mathcal{L}_q(x) &:= \ln_q n!_q - \ln_q x!_q - \ln_q (n-x)!_q \nonumber \\
&\quad + \frac{1}{2-q}\left( x^{2-q}\ln_{2-q}r + (n-x)^{2-q}\ln_{2-q}(1-r) \right) + \ln_q C_q.
\end{align}
Using the derivative approximation $\frac{d}{dx}\ln_q x!_q = \ln_q x + O(x^{-q})$, the derivative is given by:
\begin{align}
\frac{d\mathcal{L}_q(x)}{dx} &= -\ln_q x + \ln_q (n-x) + x^{1-q}\ln_{2-q}r - (n-x)^{1-q}\ln_{2-q}(1-r) + O(x^{-q}) \nonumber \\
&= \frac{1}{1-q} \left[ \left( \frac{n-x}{1-r} \right)^{1-q} - \left( \frac{x}{r} \right)^{1-q} \right] + O(x^{-q}).
\label{eq:deriv_simple}
\end{align}
For $x \leq nr$, the inequality $\frac{n-x}{1-r} \geq \frac{x}{r}$ holds.
Since the function $z \mapsto \frac{1}{1-q} z^{1-q}$ is strictly increasing for all $q \in (0, 2) \setminus \{1\}$, it directly follows that the leading term in (\ref{eq:deriv_simple}) is non-negative.
Consequently, $\frac{d\mathcal{L}_q(x)}{dx} \geq 0$ holds for sufficiently large $x$, proving the monotonicity in the asymptotic regime.
\end{proof}

The monotonicity established in Lemma \ref{lemma:monotonicity} allows us to evaluate the tail probability by the boundary term.
Utilizing this property, we derive the following Large Deviation Principle.

\begin{theorem}[Large Deviation Principle]
Let $X_{i}$ ($i=1,\dots ,n$) be Bernoulli random variables with $P(X_i=1)=r$.
If their sum follows the $q$-binomial distribution $b_{q}\left( k;n,r\right)$, then for $0 < q < 1$, the distribution satisfies the following large deviation property:
\begin{equation}
\lim_{n \to \infty} \frac{1}{n^{2-q}}\ln _{q}P\left( \frac{1}{n}\sum_{i=1}^{n}X_{i}<x\right) = -\frac{1}{2-q}D_{2-q}\left( x\left\Vert r\right. \right) .
\label{main result}
\end{equation}
\end{theorem}

\begin{proof}
Let the cumulative probability be denoted by $P_n(x) := P\left( \frac{1}{n}\sum_{i=1}^{n}X_{i}<x\right) = \sum_{k=0}^{\left\lfloor nx\right\rfloor }b_{q}\left( k;n,r\right)$.

\textit{Upper bound:}
Since $b_q(k; n, r)$ is monotonically increasing for $k \le \lfloor nx \rfloor$ (Lemma \ref{lemma:monotonicity}), the sum is bounded by:
\begin{equation}
P_n(x) \leq \left( \left\lfloor nx\right\rfloor +1\right) b_{q}\left( \left\lfloor nx\right\rfloor ;n,r\right).
\end{equation}
Taking the $q$-logarithm and dividing by $n^{2-q}$, we apply the exact pseudo-additivity $\ln_q(AB) = \ln_q B + B^{1-q}\ln_q A$.
Letting $A = \left\lfloor nx\right\rfloor +1$ and $B = b_{q}(\lfloor nx\rfloor ;n,r)$, we note that $B \le 1$.
For $0 < q < 1$, since $1-q > 0$, we have $B^{1-q} \le 1$, which strictly yields the inequality $\ln_q(AB) \le \ln_q B + \ln_q A$.
Evaluating the residual term $\frac{\ln_q A}{n^{2-q}} = \frac{O(n^{1-q})}{n^{2-q}} = O(n^{-1})$, we obtain:
\begin{equation}
\frac{1}{n^{2-q}}\ln _{q} P_n(x) \le \frac{1}{n^{2-q}}\ln _{q} b_{q}\left( \left\lfloor nx\right\rfloor ;n,r\right) + O(n^{-1}).
\end{equation}
Using Theorem \ref{thm:q-divergence_derived}, the right-hand side converges to $-\frac{1}{2-q}D_{2-q}\left( x\left\Vert r\right. \right)$.
Thus:
\begin{equation}
\limsup_{n \to \infty} \frac{1}{n^{2-q}}\ln _{q}P_n(x) \leq -\frac{1}{2-q}D_{2-q}\left( x\left\Vert r\right. \right) .
\label{q-upper}
\end{equation}

\textit{Lower bound:}
For $x<r$, we have $P_n(x) \geq b_{q}\left( \left\lfloor nx\right\rfloor ;n,r\right)$.
Thus,
\begin{equation}
\liminf_{n \to \infty} \frac{1}{n^{2-q}}\ln _{q}P_n(x) \geq \lim_{n \to \infty} \frac{1}{n^{2-q}}\ln _{q}b_{q}\left( \left\lfloor nx\right\rfloor ;n,r\right) = -\frac{1}{2-q}D_{2-q}\left( x\left\Vert r\right. \right).
\label{q-lower}
\end{equation}

Combining (\ref{q-upper}) and (\ref{q-lower}) completes the proof of (\ref{main result}).
\end{proof}

\begin{remark}
In a preliminary conference report \cite{HS14}, it was implicitly assumed that the pseudo-additivity of the $q$-logarithm preserves the standard LDP scaling for the entire range $0 < q < 2$.
The exact asymptotic analysis, however, shows that for $1 < q < 2$, the scaling factor $A^{1-q}$ in the expansion $\ln_q(AB) = \ln_q A + A^{1-q}\ln_q B$ vanishes since $A^{1-q} = O(n^{1-q})$ and $1-q < 0$.
Consequently, the standard large deviation upper bound fails, indicating that the LDP scaling is valid only for $0 < q < 1$.
This breakdown of the macroscopic large deviation scaling reflects the statistical property of heavy-tailed distributions.
This failure is specific to the large deviation regime and does not affect the local limit behavior in the central regime.
The generalized de Moivre-Laplace theorem is established in the next section to describe the fluctuation dynamics for $0 < q < 2$.
\end{remark}


\section{Generalized Central Limit Theorem via Constructive Approach}
\label{sec:q-CLT}

This section presents a derivation of the $q$-Gaussian distribution from the generalized binomial distribution.
We first utilize a continuous approximation to demonstrate why the scaling law $n^{q/2}$ is required, contrasting with the classical $n^{1/2}$ scaling, and then provide a proof using discrete asymptotic analysis, establishing the theorem within the range $0 < q < 2$.

\subsection{Continuous Approximation Analysis: The Scaling Law}
\label{subsec:heuristic}

We analyze the scaling properties of the generalized binomial distribution using a continuous approximation.
This step clarifies the necessity of the scaling factor $n^{q/2}$ in the $q$-Central Limit Theorem.
Applying the Euler-Maclaurin integral evaluation to the $q$-Stirling's formula, the $q$-logarithm of the probability distribution can be approximated by the integral form:
\begin{align}
\ln_q b_q(\tilde{x}; n, r) &= \int_{1}^{n}\ln_q u \, du - \int_{1}^{\tilde{x}}\ln_q u \, du - \int_{1}^{n-\tilde{x}}\ln_q u \, du \nonumber \\
&\quad + \frac{1}{2-q}\left( \tilde{x}^{2-q}\ln_{2-q}r + (n-\tilde{x})^{2-q}\ln_{2-q}(1-r) \right) + \ln_q C_q + O(\ln_q n).
\label{eq:continuous_approx}
\end{align}
where $\tilde{x} \in [0, n]$ is a continuous variable interpolating the discrete index $k$.
The maximum of the probability distribution, denoted by $\mu$, is determined by the stationarity condition $\frac{d}{d\tilde{x}} \ln_q b_q(\tilde{x}; n, r) = 0$.
Evaluating the derivatives, we find the unique solution:
\begin{equation}
\mu = nr,
\end{equation}
which is consistent with the mean of the standard binomial distribution ($q=1$).
To analyze the fluctuations around the mean, we perform a Taylor expansion of $\ln_q b_q(\tilde{x}; n, r)$ around $\mu = nr$:
\begin{align}
\ln_q b_q(\tilde{x}; n, r) &= \ln_q b_q(nr; n, r) + \left.
\frac{d^2}{d\tilde{x}^2} \ln_q b_q(\tilde{x}; n, r) \right|_{\tilde{x}=nr} \frac{(\tilde{x}-nr)^2}{2} + \dots
\end{align}
A computation of the second derivative yields a result:
\begin{equation}
\left.
\frac{d^2}{d\tilde{x}^2} \ln_q b_q(\tilde{x}; n, r) \right|_{\tilde{x}=nr} = -\frac{1}{n^q r(1-r)}.
\label{eq:second_derivative}
\end{equation}
Equation (\ref{eq:second_derivative}) differs from the standard case $q=1$.
For $q=1$, the denominator is $nr(1-r)$, corresponding to the variance $\sigma^2 \propto n$.
However, for $q \neq 1$, the distribution is governed by the factor $n^q$.
This implies that the standardized variable $x_k$ must be defined using the scaling $n^{q/2}$ to ensure a non-trivial asymptotic distribution.
Motivated by this observation, we define:
\begin{equation}
x_k := \frac{k - nr}{\sqrt{n^q r(1-r)}}.
\end{equation}
We now proceed to the discrete proof based on this scaling.

\subsection{Discrete Proof of the Generalized de Moivre-Laplace Theorem}
\label{subsec:rigorous_proof}

We prove the convergence of the generalized binomial distribution to the $q$-Gaussian distribution.

\begin{theorem}[Generalized de Moivre-Laplace Theorem]
\label{thm:q-deMoivreLaplace}
Let $0 < q < 2$ and $0 < r < 1$.
Consider the generalized binomial distribution $b_q(k; n, r)$ and the standardized variable $x_k$ defined by:
\begin{equation}
x_{k} := \frac{k - nr}{\sqrt{n^{q} r(1-r)}}, \quad (0 \leq k \leq n).
\label{def:standardized_x}
\end{equation}
Let $P_n^* := b_q(\lfloor nr \rfloor; n, r)$ denote the peak probability at the mean.
For any fixed constant $L > 0$, uniformly for all integers $k$ such that $|x_k| \leq L$, the $q$-logarithm of the distribution satisfies the asymptotic expansion:
\begin{equation}
\ln_q b_q(k; n, r) = \ln_q P_n^* - \frac{1}{2} x_k^2 + O\left( n^{-\frac{2-q}{2}} \right).
\end{equation}
Consequently, the probability mass function asymptotically behaves as:
\begin{equation}
b_q(k; n, r) = P_n^* \exp_q \left( - \frac{x_k^2}{2 (P_n^*)^{1-q}} \right) (1 + o(1)).
\end{equation}
\end{theorem}

\begin{proof}
Let $k = nr + \xi_n$ and $n-k = n(1-r) - \xi_n$, where the fluctuation term is given by $\xi_n := x_k \sqrt{n^q r(1-r)}$.
Note that $\xi_n$ scales as $O(n^{q/2})$.

We start from the asymptotic expansion of the $q$-logarithm of the distribution derived in Proposition \ref{prop:q-binomial_entropy} (Eq. (\ref{q-log_q-bino-coef})).
Neglecting terms of order $O(\ln_q n)$ which are subordinate to $O(n^{2-q})$, we have:
\begin{align}
\ln _{q}b_{q}\left( k;n,r\right) &= \frac{n^{2-q}}{2-q}S_{2-q}^{\text{Tsallis}}\left( \frac{k}{n}, 1-\frac{k}{n}\right) \nonumber \\
&\quad + \frac{1}{2-q}\left( k^{2-q}\ln _{2-q}r + \left( n-k\right)^{2-q}\ln _{2-q}\left( 1-r\right) \right) + \ln_q C_q + O(n^{-q}).
\label{eq:proof_expansion_start}
\end{align}
Using the explicit form of Tsallis entropy and the identity $\ln_{2-q} y = \frac{y^{q-1}-1}{q-1}$, we factorize the common prefactor $\frac{n^{2-q}}{(2-q)(1-q)}$ to obtain:
\begin{align}
\ln _{q}b_{q}\left( k;n,r\right) &= \frac{n^{2-q}}{(2-q)(1-q)} \Bigg[ 1 - \left(\frac{k}{n}\right)^{2-q}r^{q-1} - \left(1-\frac{k}{n}\right)^{2-q}(1-r)^{q-1} \Bigg] + \ln_q C_q.
\label{eq:proof_combined}
\end{align}
We analyze the asymptotic behavior of the terms in the bracket using a Taylor expansion around the mean.
The ratio $k/n$ is expanded as $k/n = r + \delta_n$, where $\delta_n := \xi_n/n = n^{\frac{q}{2}-1}\sqrt{r(1-r)}x_k$.
Since $0 < q < 2$, $\delta_n = O(n^{\frac{q}{2}-1}) \to 0$ as $n \to \infty$.
Applying the expansion $(r+\delta)^{2-q} = r^{2-q} + (2-q) r^{1-q}\delta + \frac{(2-q)(1-q)}{2}r^{-q}\delta^2 + O(\delta^3)$, we evaluate the components in the bracket:
\begin{align}
\left(\frac{k}{n}\right)^{2-q} r^{q-1} &= r + (2-q)\delta_n + \frac{(2-q)(1-q)}{2r}\delta_n^2 + O(\delta_n^3), \\
\left(1-\frac{k}{n}\right)^{2-q} (1-r)^{q-1} &= (1-r) - (2-q)\delta_n + \frac{(2-q)(1-q)}{2(1-r)}\delta_n^2 + O(\delta_n^3).
\end{align}
Substituting these expansions back into the bracketed term in (\ref{eq:proof_combined}), we observe that the zeroth-order terms analytically cancel out ($1 - r - (1-r) = 0$), as do the first-order terms.
The leading non-vanishing contribution comes strictly from the second-order terms:
\begin{align}
1 - \left(\frac{k}{n}\right)^{2-q}r^{q-1} - \left(1-\frac{k}{n}\right)^{2-q}(1-r)^{q-1} &= - \frac{(2-q)(1-q)}{2} \delta_n^2 \left( \frac{1}{r} + \frac{1}{1-r} \right) + O(\delta_n^3) \nonumber \\
&= - \frac{(2-q)(1-q)}{2 r(1-r)} \delta_n^2 + O(\delta_n^3).
\end{align}
Multiplying this result by the prefactor $\frac{n^{2-q}}{(2-q)(1-q)}$, and noting that $O(n^{2-q} \delta_n^3) = O(n^{-\frac{2-q}{2}})$, we obtain:
\begin{align}
\ln _{q}b_{q}\left( k;n,r\right) &= - \frac{n^{2-q}}{2 r(1-r)} \delta_n^2 + \mathcal{K}_n(q) + O\left(n^{-\frac{2-q}{2}}\right),
\end{align}
where $\mathcal{K}_n(q)$ aggregates all terms independent of the fluctuation $x_k$, including the normalization $\ln_q C_q$.
Substituting $\delta_n^2 = n^{q-2}r(1-r)x_k^2$, the quadratic term evaluates to $-\frac{1}{2}x_k^2$.
To determine $\mathcal{K}_n(q)$, we evaluate this equation at the mean $k = nr$, which corresponds to $x_k = 0$.
At this point, the left-hand side is $\ln_q b_q(nr; n, r) = \ln_q P_n^*$, and the quadratic term on the right-hand side vanishes.
This yields the identity $\mathcal{K}_n(q) = \ln_q P_n^*$.
Thus, we establish the exact asymptotic form:
\begin{equation}
\ln _{q}b_{q}\left( k;n,r\right) = \ln_q P_n^* - \frac{1}{2} x_k^2 + O\left(n^{-\frac{2-q}{2}}\right).
\label{eq:exact_q_log_clt}
\end{equation}

To recover the probability distribution, we apply the algebraic identity of the $q$-logarithm, $\ln_q y = \ln_q c - X \iff y = c \exp_q\left(- X / c^{1-q}\right)$.
Setting $y = b_q$, $c = P_n^*$, and $X = \frac{1}{2}x_k^2$, we arrive at:
\begin{equation}
b_q(k; n, r) = P_n^* \exp_q \left( - \frac{x_k^2}{2 (P_n^*)^{1-q}} \right) (1 + o(1)).
\label{eq:final_q_gaussian}
\end{equation}

Finally, to relate the discrete probability mass function $b_q(k; n, r)$ to the continuous probability density function $\mathcal{G}_q(x)$, we account for the grid spacing $\Delta x_n = 1/\sqrt{n^q r(1-r)}$.
Equation (\ref{eq:final_q_gaussian}) establishes that the step-function density $b_q(k; n, r) / \Delta x_n$ converges pointwise to the $q$-Gaussian kernel.
By Scheff\'{e}'s theorem \cite{Billingsley1995}, since the discrete probabilities sum to $1$, the pointwise convergence of these normalized densities guarantees convergence in distribution:
\begin{equation}
    \lim_{n \to \infty} \frac{b_q(k; n, r)}{\Delta x_n} = \lim_{n \to \infty} \sqrt{n^q r(1-r)} \, b_q(k; n, r) = \mathcal{G}_q(x),
\end{equation}
where $\mathcal{G}_q(x)$ is the continuous $q$-Gaussian density.
This completes the proof.
\end{proof}

\begin{remark}[Contrast with the Large Deviation Regime]
As observed in Section IV, the macroscopic scaling of the large deviation principle fails for the regime $1 < q < 2$ due to the heavy-tailed nature of the probability distribution.
The present theorem demonstrates that, in contrast to the global LDP, the local limit behavior remains universally valid across the entire range $0 < q < 2$ under the appropriate fluctuation scaling $n^{q/2}$.
This confirms that the $q$-Gaussian distribution is the local attractor for these generalized statistics, even in the regime where the macroscopic exponential bounds break down.
\end{remark}



\section{Numerical Verification}
\label{sec:numerical_verification}

To validate the asymptotic results derived in Theorem \ref{thm:q-deMoivreLaplace}, we performed numerical evaluations of the generalized binomial distribution for large system sizes ($n=50,000$ and $n=500,000$).
Unlike Monte Carlo simulations that rely on random sampling, we computed the probability mass functions by utilizing the algebraic properties of the $q$-logarithm, ensuring precision without numerical overflow.
In accordance with our constructive approach, we applied an algebraic shift based on the maximum weight followed by standard linear normalization ($\sum b_q = 1$), avoiding the use of escort distributions.
The horizontal axis was standardized using the scaling factor derived in our theoretical analysis:
\begin{equation}
    x = \frac{k - \mu}{\sigma_q}, \quad \text{where} \quad \sigma_q = n^{q/2} \sqrt{r(1-r)}.
\label{eq:scaling_verification}
\end{equation}
This macroscopic scaling $n^{q/2}$ is mathematically required to ensure convergence to a non-trivial asymptotic distribution and to compensate for the extreme variance expansion inherent in the non-additive regime.

\subsection{Standard Convergence Regimes ($q < 5/3$)}

We first verify the convergence in the regimes where the limit distribution possesses a finite variance (for $1 \le q < 5/3$) or compact support (for $0 < q < 1$).
We selected two representative values: $q=1.5$ (heavy-tailed regime) and $q=0.5$ (compact support regime) with $r=0.5$.
Fig. \ref{fig:q-CLT_main_linear} and Fig. \ref{fig:q-CLT_main_log} present the comparison between the discrete distributions (shaded areas) and the theoretical limits $\mathcal{G}_q(x)$ (dashed curves) in linear and logarithmic scales, respectively.

\begin{figure}[!htbp]
    \centering
    \includegraphics[width=\columnwidth]{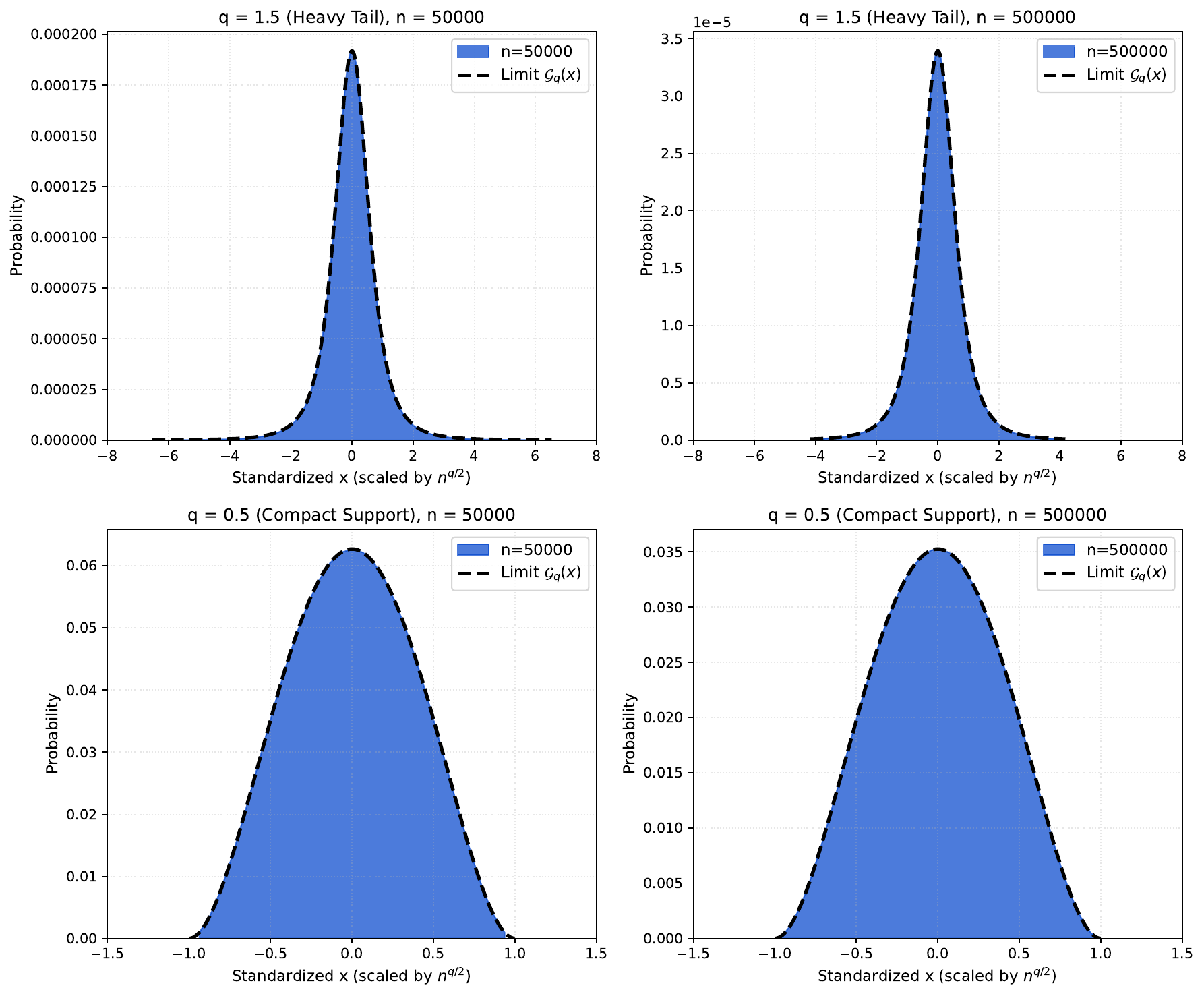}
    \caption{Numerical verification of the generalized de Moivre-Laplace theorem (linear scale) for $n=50,000$ and $n=500,000$ with $r=0.5$.
    The empirical probability mass functions of the $q$-binomial distribution (shaded areas) are plotted against the standardized variable $x = (k-nr)/n^{q/2}$.
    The dashed black lines represent the asymptotic continuous $q$-Gaussian limits $\mathcal{G}_q(x)$. 
    Top row: $q=1.5$ (heavy-tailed regime). Bottom row: $q=0.5$ (compact support regime).
    The linear scale shows the geometric convergence of the probability core around the mean.}
    \label{fig:q-CLT_main_linear}
\end{figure}

\begin{figure}[!htbp]
    \centering
    \includegraphics[width=\columnwidth]{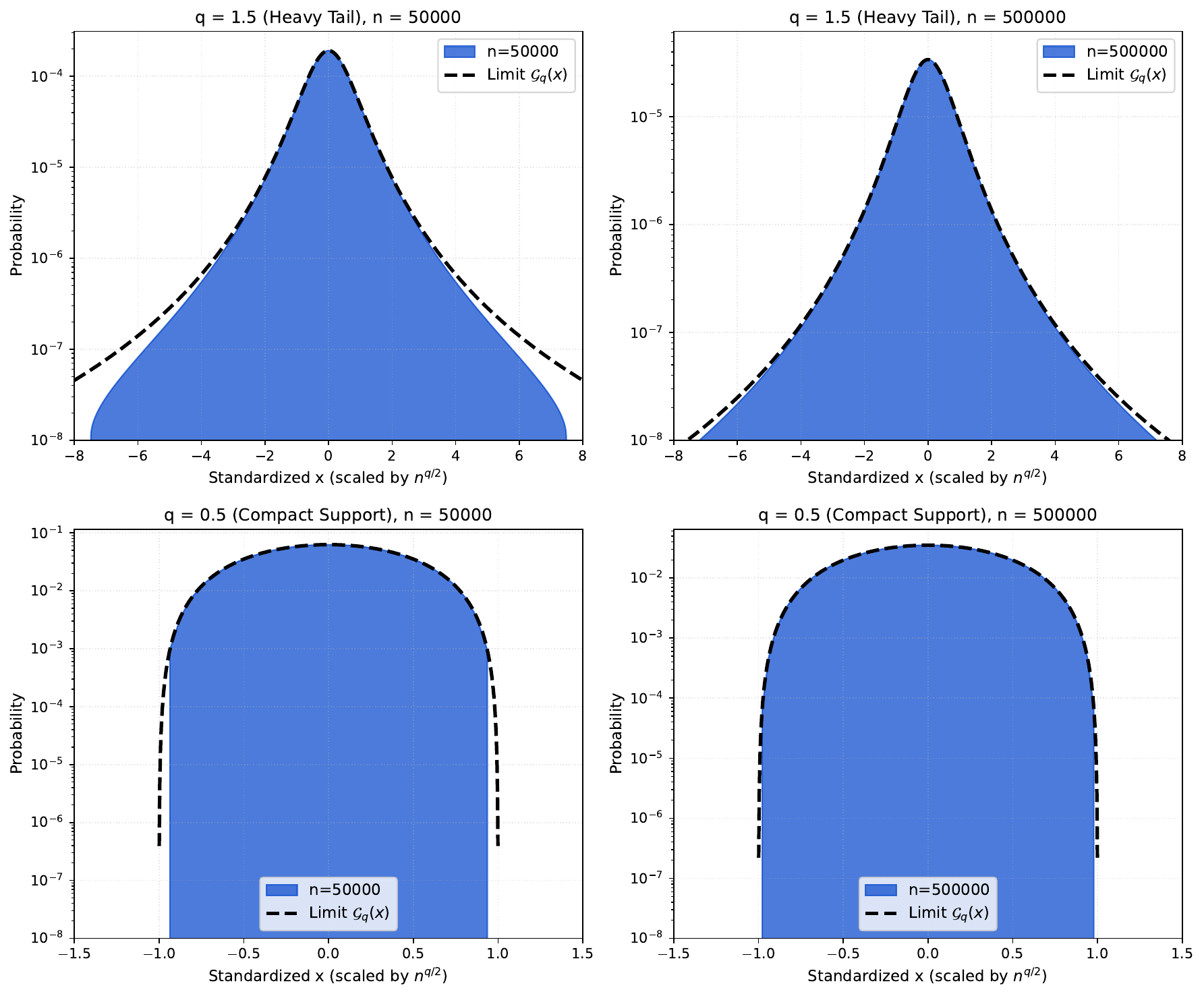}
    \caption{Logarithmic scale representation of the distributions shown in Fig. \ref{fig:q-CLT_main_linear}.
    Top row ($q=1.5$): The logarithmic axis visualizes the power-law tail behavior, showing that the tails of the discrete distribution converge to the $q$-Gaussian profile.
    Bottom row ($q=0.5$): The logarithmic scale demonstrates the adherence to the finite boundary cutoff inherent in the compact support regime.}
    \label{fig:q-CLT_main_log}
\end{figure}

The logarithmic representation (Fig. \ref{fig:q-CLT_main_log}) shows that as $n$ increases, the tail of the empirical distribution converges to the $q$-Gaussian profile. This alignment demonstrates that the macroscopic $q$-Gaussian distribution is generated from the combinatorial trials.

\subsection{Divergent Variance Regime ($q \ge 5/3$)}

To examine the scaling $n^{q/2}$, we consider the regime where $q \ge 5/3$.
In this domain, the second moment of the theoretical $q$-Gaussian distribution diverges, rendering the classical Central Limit Theorem inapplicable.
Fig. \ref{fig:q-CLT_divergent_linear} and Fig. \ref{fig:q-CLT_divergent_log} display the results for $q=1.8$.

\begin{figure}[!htbp]
    \centering
    \includegraphics[width=\columnwidth]{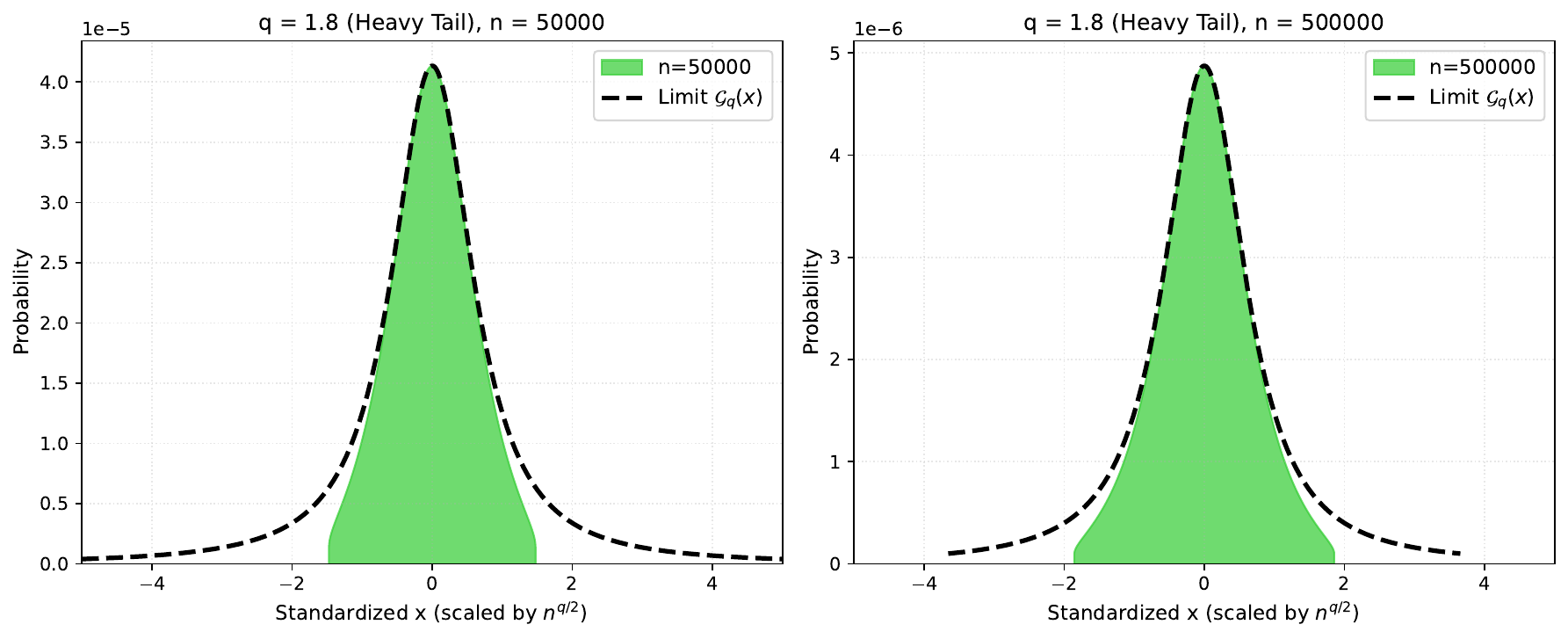}
    \caption{Numerical verification for the divergent variance regime ($q=1.8$) in linear scale.
    Despite the infinite theoretical variance, the generalized scaling factor $n^{q/2}$ absorbs the anomalous fluctuations, preserving the geometrical shape of the distribution within the finite observation window.}
    \label{fig:q-CLT_divergent_linear}
\end{figure}

\begin{figure}[!htbp]
    \centering
    \includegraphics[width=\columnwidth]{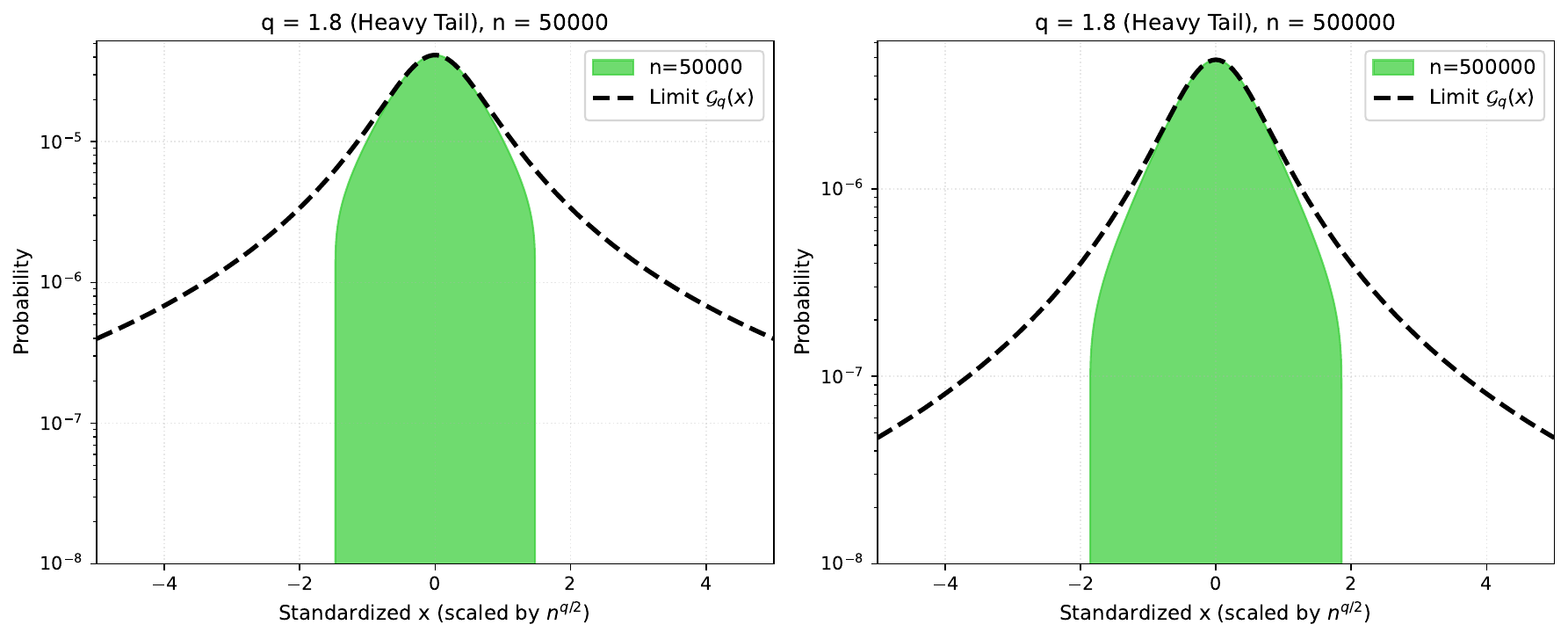}
    \caption{Logarithmic scale representation for $q=1.8$.
    The plot shows the heavy-tail characteristics of the generalized binomial process. The continuous $q$-Gaussian limit captures the asymptotic behavior of these far-tail events.}
    \label{fig:q-CLT_divergent_log}
\end{figure}

The numerical results for $q=1.8$ confirm that the $q$-Gaussian distribution remains the local attractor. The scaling law $n^{1.8/2}$ dynamically compensates for the extreme variance expansion, maintaining the structural integrity of the generalized binomial distribution.
This validates our formulation across the entire physical range of the deformation parameter $0 < q < 2$.


\section{Conclusion}
\label{sec:conclusion}

In this paper, we have established a mathematical framework for power-law distributions, deriving the generalized de Moivre-Laplace theorem from the algebraic structure of the nonlinear differential equation $dy/dx=y^q$.
Unlike approaches that rely on variational principles or functional assumptions, our framework proceeds constructively from the $q$-Stirling's formula to the $q$-binomial distribution, and finally to the generalized limit theorems.
This work introduces a generative paradigm for modeling power-law phenomena.
While $q$-Gaussian distributions have been widely applied to empirical data, their use has often lacked a fundamental generative mechanism.
The generalized binomial distribution established here provides a microscopic model.
Our framework connects microscopic probabilistic trials directly to macroscopic heavy-tailed fluctuations, enabling the constructive modeling of stochastic processes grounded in exact convergence theorems.
Our asymptotic analysis yielded two primary results that link algebra, probability, and information geometry.
First, via the Large Deviation Principle, we proved that for $0 < q < 1$, the $\alpha$-divergence is identified as the exact rate function governing the generalized binomial process, while demonstrating the fundamental breakdown of this macroscopic scaling for heavier tails ($q > 1$).
This confirms that the algebraic structure of the $q$-product is compatible with the geometry of information.
Second, via the generalized de Moivre-Laplace theorem, we rigorously derived the $q$-Gaussian distribution as the universal local attractor for the entire range $0 < q < 2$.
We identified the fluctuation scaling law $n^{q/2}$ as a strict mathematical requirement to ensure a non-trivial asymptotic distribution.
This generalized scaling $n^{q/2}$ provides a physical interpretation of the deformation parameter $q$ as the scaling exponent governing anomalous fluctuations.

Beyond the probabilistic limit theorems established in this paper, the proposed constructive framework provides a mathematical foundation for broader theoretical applications. 
While the present work focuses on the generalized de Moivre-Laplace theorem and the Large Deviation Principle, the algebraic structure of the $q$-logarithm naturally extends to higher-order fluctuation analyses. 
Specifically, in the vicinity of the classical limit $q \to 1$, the parameter $q$ systematically governs the variance of information density (varentropy), which plays a critical role in finite-size scaling and non-equilibrium fluctuations.

The detailed operational interpretation of this varentropy-driven scaling, its geometric properties on statistical manifolds, and its direct applications to finite blocklength information theory are beyond the scope of this fundamental probabilistic framework. 
These information-theoretic and geometric extensions are thoroughly addressed in the subsequent papers of this series \cite{Hexalogy_PartII, Hexalogy_PartVI}. 
Ultimately, this constructive approach aims to unify the distinct mathematical languages of probability, statistical mechanics, and information geometry into a single, cohesive framework.

\begin{acknowledgments}
The first author (H.S.) expresses his deepest gratitude to Prof. Emeritus Jan Naudts (University of Antwerp) for his invaluable discussions, guidance, and warm hospitality during a visiting research stay from April to September 2013. The profound insights gained during this period laid the foundational groundwork for this long-term research program.
H.S. also gratefully acknowledges the hospitality and support of Politecnico di Torino, Italy, during his subsequent sabbatical stay from October 2013 to March 2014, where the core concepts of this specific work were initially conceived.
The authors are also grateful to Prof. Ugur Tirnakli for his constructive suggestions on the numerical verification and the presentation of the figures.
\end{acknowledgments}

\bibliography{aipsamp}

\end{document}